\documentclass[11pt]{article}

\usepackage[utf8]{inputenc}
\usepackage{graphicx}
\usepackage{amsthm,amssymb}
\usepackage{graphicx}
\usepackage{fontenc}
\usepackage{xcolor}
\usepackage{dsfont}
\usepackage{amssymb}
\usepackage{enumitem}
\usepackage{ragged2e}
\usepackage{mathtools}
\usepackage{fullpage}
\usepackage{float}

\restylefloat{figure}

\newcommand{\purple}{\color{purple}}

\newcommand{\N}{\mathbb{N}}
\newcommand{\R}{\mathbb{R}}

\def\XXint#1#2#3{{\setbox0=\hbox{$#1{#2#3}{\int}$ }
\vcenter{\hbox{$#2#3$ }}\kern-.6\wd0}}

\newcommand{\norm}[1]{\left\|#1\right\|}
\newcommand{\parentheses}[1]{\left(#1\right)}
\newcommand{\brackets}[1]{\left[#1\right]}
\newcommand{\module}[1]{\left|#1\right|}
\newcommand{\braces}[1]{\left\{#1\right\}}
\renewcommand{\date}[1]{\setcounter{equation}{0} \vspace{1cm}\begin{center} {\Large\bf #1} \end{center}}

\counterwithin{equation}{section}

\newtheorem{statement}{Statement}[section]

\newtheorem{theorem}[statement]{Theorem}
\newtheorem*{conjecture*}{Conjecture}
\newtheorem{corollary}[statement]{Corollary}
\newtheorem{lemma}[statement]{Lemma}

\theoremstyle{definition}
\newtheorem{definition}[statement]{Definition}

\theoremstyle{remark}
\newtheorem{remark}[statement]{Remark}
\newtheorem{example}[statement]{Example}

\title{On average population levels for models with directed diffusion in heterogeneous environments}
\author{ André Rickes and Elena Braverman \\
Dept. of Math. and Stats., University of
Calgary, \\ 2500 University Drive N.W., Calgary, AB  T2N 1N4, Canada}

\begin{document}
\maketitle

\begin{abstract}
In 2006 (J. Differential  Equ.), Lou proved that, once the intrinsic growth rate $r$ in the logistic model is proportional to the  spatially heterogeneous carrying capacity $K$ ($r=K^1$), the total population under the regular diffusion exceeds the total of the carrying capacity.
DeAngelis et al (J. Math. Biol. 2016) argued that the prevalence of the population over the carrying capacity is only observed when   the growth rate and the carrying capacity are positively correlated, at least for slow dispersal. Guo et al (J. Math. Biol. 2020) justified that, once $r$ is constant ($r=K^0$), the total population is less than the cumulative carrying capacity. Our paper fills up the gap for when $r=K^{\lambda}$ for any real $\lambda$, disproving an assumption that there is a critical $\lambda^{\ast} \in (0,1)$ at which the tendency of the prevalence of the carrying capacity over the total population size changes, demonstrating instead that the relationship is more complicated.
In addition, we explore the dependency of the total population size on the diffusion coefficient when the third parameter of the dispersal strategy $P$ is involved: the diffusion term is $d \Delta(u/P)$, not just $d \Delta u$, for any $\lambda$. We outline some differences from the random diffusion case, in particular, concerning the profile of the total population as a function of the diffusion coefficient.


{\bf AMS subject classification:} 92D25, 35J25, 92D40, 35J60, 35B40
\\
{\bf Keywords:} spatial heterogeneity, random and directed diffusion, average population level, mathematical ecology
\end{abstract}























%

%


\section{Introduction}

The idea to include diffusion for a more comprehensive description of dispersal in population ecology was a significant breakthrough in analysing invasions and spatial species interactions \cite{Skellam}. 
Together with random walks approach, application of diffusion became a cornerstone in analysing populations movements and their effect \cite{Ang2016,Ang2016_a,Hastings,Levin}.
Since different species disperse in unique ways, their specific dispersal mechanisms must be accounted for when incorporating spatial heterogeneity into time-dependent mathematical models \cite{CCBook}.
The effects of spatial heterogeneity were explored in \cite{He1,He2,He3}. The environment variability was strongly connected to the notion of the carrying capacity which was actively discussed \cite{Mallet,Zhang}.

This paper studies the relationship between the carrying capacity of the environment and the total population size of a single species adopting different dispersal strategies. This connection was explored in  \cite{Ang2016,Ang2016_a,Guo_2020,Lou}, where population dynamics was modelled using the logistic equation with random diffusion 
\begin{equation}\label{intro1}
    \begin{cases}
        \frac{\partial u}{\partial t}(x,t)=d\Delta u(t,x) +r(x)u (t,x) \parentheses{1-\frac{u(t,x)}{K(x)}}, &  x \in \Omega, ~~ t>0 \\
        u(x,0)=u_0(x),& x\in\Omega \\
        \frac{\partial u}{\partial n}(t,x)=0 , &  x \in \partial\Omega,~~ t>0.
    \end{cases} 
\end{equation}
Here $r(x)$ and $K(x)$ are time-independent positive on $\bar\Omega$ functions representing the intrinsic growth rate and the carrying capacity, respectively. 

The disadvantage of considering that species diffuse randomly as in \eqref{intro1} is that, as the diffusion becomes more important in our model (when $d$ becomes large enough), the solution of \eqref{intro1.5}  goes away from the carrying capacity profile \cite{C1,C4,C2,C3}. 
This leads to the counterintuitive outcome that in the competition of several species identical in all the parameters but the value of the diffusion coefficient $d$ in \eqref{intro1},
the slower diffuser wins \cite{Dockery}. 
Multiple studies in mathematical biology indicated that it is reasonable to consider that populations diffuse in an ideal free dispersal framework, characterized by assuming populations can diffuse as they wish, but they choose to disperse from places with lower suitability to places where their chances of survival increase \cite{Bonte,CC1,CCL1}. Populations choosing an ideal free distribution are found to be harder to be invaded by competing species sharing their habitat \cite{CAMWA2016,C1,CCL2,Korobenko2012,Korobenko2014}, and therefore, a directed diffusion based on the heterogeneity of the environment becomes more present on the modelling of population dynamics, see \cite{CC_2025,C4} for a detailed overview. 

Let $u_d$ denote a stationary solution to \eqref{intro1}, i.e., let $u_d$ be a solution to
\begin{equation}\label{intro1.5}
    \begin{cases}
 \displaystyle             d\Delta u(x) +r(x)u (x) \parentheses{1-\frac{u(x)}{K(x)}}=0, &  x \in \Omega,\\
   \displaystyle           \frac{\partial u}{\partial n}(x) , &  x \in \partial\Omega.
    \end{cases} 
\end{equation}
Then, as the diffusion coefficient tends to zero, $u_d$ tends to the carrying capacity at every point of the domain $\Omega$,
and as $d \to +\infty$, the stationary solution $u_d$ is distributed uniformly over $\Omega$ \cite{Ang2016,Lou}. 
However, a problem still under investigation is the relation between the total stationary population for $0<d<+\infty$ and the total carrying capacity. 

Lou \cite{Lou} demonstrated that for 
$r(x) = \alpha K(x)$ for some constant $\alpha>0$, the total population size of the solution $u_d$ to \eqref{intro1.5} exceeds the total carrying capacity for every $d\in(0,+\infty)$, i.e.,
\begin{equation}\label{intro1a}
\int_{\Omega} u_d~ dx > \int_{\Omega} K~dx.
\end{equation}
On the other hand, in \cite{Guo_2020}, the authors justified that when $r(x)$ is constant, and therefore is independent of $K (x)$,  the logistic equation with spatially heterogeneous resources always supports a total population strictly smaller than the total
carrying capacity at equilibrium, which is just opposite to the case $r = \alpha K$. This result can be interpreted as, once the growth rate of species is unchanged over the domain, a uniform distribution of resources maximizes the population of species. To strengthen this result, Mazari-Fouquer \cite{MazariFouquer2024} quantifies the optimality of a constant carrying capacity on the maximization of average population of species, compared to other resource distribution profiles.

DeAngelis {\em et al} \cite{Ang2016} argued that, once the growth rate $r$ is not proportional to $K$, we can get different relations between the total population and the total carrying capacity. In particular, once these two functions are positively correlated (meaning that $r(x) = h(K(x))$ with $h'>0$), the total population of $u_d$ exceeds the total carrying capacity, while
for negatively correlated $r$ and $K$ ($r(x) = h(K(x))$ and $h'<0$), the total population of $u_d$  is less for slow dispersal. The  proofs of these results in \cite{Ang2016}  are strongly based on the estimation 
of the integral of $u_d$ for small $d>0$. 

Unlike this local approach, in \cite{Lou} inequality \eqref{intro1a} was justified globally (and the opposite in \cite{Guo_2020} for the constant growth rate), for any $d>0$.
In addition, it was noticed that the total solution equals the total of the carrying capacity when either $d=0$ or $d \to \infty$ when $r = \alpha K$. 
Then, a maximum of the average should be attained for some $d \in (0,\infty)$.
Although \cite{Lou} conjectures that
\begin{equation}
\label{def:integral}
M(d) = \int_{\Omega} u_d~ dx
\end{equation}
attains exactly one local maximum in $(0,+\infty)$, it was proven by \cite{Liang2012} the existence of a carrying capacity profile that induces multiple local maxima for the average of $u_d$.

Connecting the results of \cite{Lou}  and \cite{Guo_2020}, we notice that for $r=\alpha K=\alpha K^1$, the total population $u_d$ exceeds the total carrying capacity, 
while for $r=\alpha =\alpha K^0$ it is less. None of the two papers explores the case $r=\alpha K^{\lambda}$, $\lambda \in (0,1)$,
and the current paper fills this gap. The naive suggestion that there is some critical $\lambda^{\ast}  \in (0,1)$  such that for larger values \eqref{intro1a} holds,
while the opposite inequality is satisfied for smaller $\lambda$ is incorrect, as the situation is more complicated.
Moreover, we consider any $\lambda$, not only in $[0,1]$.

In addition, we explore the model incorporating the third space-dependent parameter associated with dispersal. 
A strategy used to alleviate the effect of random diffusion with directed advection suggested in \cite{Brav} assumed that we can model the diffusing quantity in \eqref{intro1} as not $u$, but $u/P$, where $P$ is the dispersal strategy adopted by the species. Then \eqref{intro1.5} is  replaced by
\begin{equation}\label{intro2}
    \begin{cases}
  \displaystyle       d\Delta\parentheses{\frac{u}{P}}+ru\parentheses{1-\frac{u}{K}}=0, & x \in \Omega,  \\
  \displaystyle            \frac{\partial}{\partial n}\parentheses{\frac{u}{P}}=0,  & x \in \partial\Omega.
    \end{cases} 
\end{equation} 
If $P$ is proportional to $K$, the carrying capacity is a solution of \eqref{intro2}, independently of $d$, and species will evolve based on the available resources.
Several weighted inequalities connecting $r,K,P$ with the stationary solution $u_d$ were obtained in \cite{CAMWA2016,Korobenko2012,Korobenko2014}
as auxiliary statements on the way of demonstrating that $P=K$ in \eqref{intro2} is a winning strategy, allowing the species choosing it either to protect its own habitat or invade a habitat of a competitor  choosing a different way to disperse, or considering two different species cooperating in dispersal strategies in a sense that $K$ is a linear combination of two dispersal strategies with positive coefficients. However, explicit relations as in \eqref{intro1a} have never been considered, and in this paper, we adapt the strategies from \cite{Ang2016} to obtain the desired estimates. If several species coexist choosing different dispersal strategies, this can lead to coexistence when each of the populations occupies its niche in the spatial distribution of the available resources \cite{Averill,CAMWA2016}.
If we relate the results of \cite{Dockery} to those in \cite{Lou}, this does not give an advantage to a competitor achieving higher density in the absence of others.
Otherwise, an intermediate diffusion coefficient where the maximum density is attained would be optimal.


The purpose of this paper is two-fold:
\begin{enumerate}
\item
We analyze the behaviour of the total population $M(d)$ when the intrinsic growth rate of species is described by $r=\alpha\parentheses{\frac{K}{P}}^\lambda$.
In particular, when $P \equiv 1$ as in \eqref{intro1}, these results fill the gap between the results of \cite{Lou}  and \cite{Guo_2020}.
\item
For the model with three parameters $P,r,K$ as in \eqref{intro2}, we explore the relation of the total population and the total carrying capacity, coinciding with \eqref{intro1a}
or the opposite inequality, and establish sufficient conditions for both. Overall, we investigate  the dependency of the total  population $M(d)$ on $d$;
\end{enumerate}

In particular, we answer the following questions.

\begin{enumerate}
\item
Are there specific choices of the dispersal strategy $P$ in \eqref{intro2} that guarantee \eqref{intro1a}?
\item Since \eqref{intro1a} holds for every $d\in(0,+\infty)$ when $r=\alpha\frac{K}{P}$, what happens when $r=\alpha\parentheses{\frac{K}{P}}^\lambda$?

\end{enumerate}

Further, let us describe the structure of the paper. In Section~\ref{sec:auxil} we justify the solution limits for infinitesimal or very large diffusion coefficients.
Section~\ref{sec:main} contains the main results of the paper.
In particular, we justify under which conditions \eqref{intro2} is satisfied for any $d$, and explore the situations when either  \eqref{intro2}  or the opposite inequality are satisfied for small $d$. In Section~\ref{sec:examples}, we provide examples illustrating the main results and simulating the cases when the theoretical results fail to establish the relations between the total population and the total carrying capacity. 
Section~\ref{sec:discussion} contains discussion of both obtained results and remaining open questions. Appendix includes some additional weighted inequalities connecting the total population and the total carrying capacity of the environment.


\section{Limit Population Sizes for Very Slow or Very Fast Diffusion}
\label{sec:auxil}

Let $\Omega$ be a bounded open region of $\R^n$. In this paper, we assume that\\

\noindent (A) \,\,$r$, $K$, and $P$ are non-constant, positive functions over $\overline{\Omega}$ such that $r$, $K$, $P\in C^2(\overline{\Omega})$.\\ 

Let us denote the Sobolev space $W^{1,2}(\Omega)=\braces{f\in L^1(\Omega):\,\,\,\nabla f\in L^2(\Omega)}$, where  $\nabla f$ is considered in the weak sense, equipped with the norm
\begin{equation*}
    \norm{f}_{W^{1,2}(\Omega)}=\parentheses{\norm{f}_{L^2(\Omega)}+\norm{\nabla f}_{L^2(\Omega)}}^\frac{1}{2}.
\end{equation*}

Denote by $u_d$ a solution for \eqref{intro2} for each $d>0$.

\begin{lemma}\label{Prop_1_DeAngelis}
    As $d\to0^+$, we obtain
    \begin{equation*}
         u_d\to K\hspace{0.5cm} in \hspace{0.5cm} L^\infty(\overline{\Omega})\cap W^{1,2}(\Omega).
    \end{equation*}
\end{lemma}

\begin{proof}

    First we show convergence in $L^\infty(\Omega)$. For that purpose, we show that, for $d>0$ small, we obtain $\norm{u_d-K}_{L^\infty(\Omega)}\leq \parentheses{\alpha \max_{x\in\overline{\Omega}}P(x)}d$, where
    \begin{equation*}
        \alpha =2\frac{\norm{\Delta(\frac{K}{P})}_{L^\infty(\overline{\Omega})}}{\min_{x\in\overline{\Omega}}(r(x)P(x))}.
    \end{equation*}
    Denote the operator
    \begin{equation*}
        Lu=d\Delta\parentheses{\frac{u}{P}}+ru\parentheses{1-\frac{u}{K}}.
    \end{equation*}
    For each $d>0$, let $\overline{u}(x)=K(x)+\alpha P(x)d$, and compute
    \begin{align*}
        L\overline{u}=&d\Delta\parentheses{\frac{K}{P}}+d\Delta(\alpha d)+r(K+\alpha Pd)\parentheses{-\frac{\alpha Pd}{K}}
\\  = &d\brackets{\Delta\parentheses{\frac{K}{P}}-\alpha rP\parentheses{1+\frac{\alpha Pd}{K}}}
\\
        \leq  & d\brackets{\Delta\parentheses{\frac{K}{P}}-\alpha rP}\\
        =&d\brackets{\Delta\parentheses{\frac{K}{P}}-2\frac{\norm{\Delta(\frac{K}{P})}_{L^\infty(\overline{\Omega})}}{\min_{x\in\overline{\Omega}}(r(x)P(x))}rP}
\\  \leq &  d\brackets{\Delta\parentheses{\frac{K}{P}}-2\norm{\Delta\parentheses{\frac{K}{P}}}_{L^\infty(\overline{\Omega})}}\leq0.
    \end{align*}
    Therefore, $\overline{u}$ is an upper solution for $Lu=0$.

    If $\underline{u}(x)=K(x)-\alpha P(x)d$, we see that $\underline{u}>0$ once $d>0$ is small ($0<d< \min_{\overline{\Omega}}(K/P)/(2\alpha)$), and
    \begin{equation*}
        L\underline{u}=d\Delta\parentheses{\frac{K}{P}}+r(K-\alpha Pd)\frac{\alpha Pd}{K}=d\brackets{\Delta\parentheses{\frac{K}{P}}+\frac{1}{2}\alpha rP+\alpha rP\parentheses{\frac{1}{2}-\frac{\alpha Pd}{K}}},
    \end{equation*}
    and since $\frac{1}{2}-\frac{\alpha Pd}{K}>0$ for $0<d< \min_{\overline{\Omega}}(K/P)/(2\alpha)$, we obtain
     \begin{align*}
      L\underline{u} & \geq d\brackets{\Delta\parentheses{\frac{K}{P}}+\frac{1}{2}2\frac{\norm{\Delta(\frac{K}{P})}_{L^\infty(\overline{\Omega})}}{\min_{x\in\overline{\Omega}}(r(x)P(x))}rP}
\\ & \geq d\brackets{\Delta\parentheses{\frac{K}{P}}+\norm{\Delta \left( \frac{K}{P} \right)}_{L^\infty(\overline{\Omega} )}}\geq 0,
    \end{align*}
    and we conclude that  $\underline{u}$ is a lower solution for $Lu=0$. And since $\underline{u}<\overline{u}$ in $\Omega$ and $u_d$ is the solution to $Lu=0$, we obtain $\underline{u}<u_d<\overline{u}$ in $\Omega$, which implies
    \begin{equation*}
        K(x)-\alpha P(x)d<u_d(x)<K(x)+\alpha P(x)d \implies |u_d(x)-K(x)|<\alpha P(x)d\leq\parentheses{\alpha \max_{x\in\overline{\Omega}}P(x)}d,
    \end{equation*}
    for any $x\in\Omega$ and $d>0$ small. Therefore, $u_d\to K$ in $L^\infty(\Omega)$ as $d\to 0^+$, and since $u_d$ is bounded and continuous in $\overline{\Omega}$ we also obtain convergence in $L^\infty(\overline{\Omega})$.

    Finally, to show convergence on $W^{1,2}(\Omega)$, set $\varphi(x)=u_d-K(x)$ and notice that $\varphi$ satisfies
    \begin{equation}\label{Eq.1 Feb25th}
        d\Delta\parentheses{\frac{\varphi}{P}}+d\Delta\parentheses{\frac{K}{P}}+\frac{r}{K} u_d(-\varphi)=0.
    \end{equation}

    Multiplying \eqref{Eq.1 Feb25th} by $\frac{\varphi}{Pd}$ and integrating, we obtain 
    \begin{align*}
        &\int_\Omega\Delta\parentheses{\frac{\varphi}{P}}\frac{\varphi}{P}\,dx+\int_\Omega\Delta\parentheses{\frac{K}{P}}\frac{\varphi}{P}\,dx - \int_\Omega \frac{r}{K}\frac{u_d\varphi^2}{Pd}\,dx=0\\
        \implies& -\int_\Omega\module{\nabla\parentheses{\frac{\varphi}{P}}}^2\,dx+\int_\Omega\Delta\parentheses{\frac{K}{P}}\frac{\varphi}{P}\,dx - \int_\Omega \frac{r}{K}\frac{u_d\varphi^2}{Pd}\,dx=0
    \end{align*}
    and thus, we conclude that
    \begin{equation*}
        \int_\Omega\module{\nabla\parentheses{\frac{\varphi}{P}}}^2\,dx=\int_\Omega\Delta\parentheses{\frac{K}{P}}\frac{\varphi}{P}\,dx-\int_\Omega\frac{r}{K}\frac{u_d\varphi^2}{dP}\,dx\leq\int_\Omega \varphi\norm{\frac{1}{P}\Delta\parentheses{\frac{K}{P}}}_{L^\infty(\Omega)}\,dx,
    \end{equation*}
    which tends to 0 as $d\to0^+$, since $\varphi=u_d-K\to0$ in $L^\infty(\Omega)$. Therefore, as $d\to0^+$, 
    \begin{equation*}
        \int_\Omega\module{\nabla\parentheses{\frac{\varphi}{P}}}^2\,dx\to0,
    \end{equation*}
    and since $\frac{1}{P}$, $\nabla\parentheses{\frac{1}{P}}$ are bounded on $\bar\Omega$, we conclude that $u_d\to K$ in $W^{1,2}(\Omega)$ as $d\to0^+$.
\end{proof}

\begin{lemma}\label{Prop 2 DeAngelis}
    As $d\to+\infty$, we get
        \begin{equation*}
            u_d\to \beta P(x)\hspace{0.5cm} in \hspace{0.5cm} L^\infty(\overline{\Omega}),
        \end{equation*}
    where $\displaystyle \beta=\frac{\int_\Omega rP\,dx}{\int_\Omega\frac{r}{K}P^2\,dx}$.

\end{lemma}

\begin{proof}
    
    Integrating \eqref{intro2} over $\Omega$ and using the boundary condition $\frac{\partial}{\partial n}(\frac{u_d}{P})=0$ on $\partial\Omega$, we obtain
    \begin{equation}\label{Eq.2 Feb25th}
        \int_\Omega ru_d\parentheses{1-\frac{u_d}{K}}\,dx=0.
    \end{equation}
    Now, since $u_d$ and $\nabla u_d$ are bounded in $L^\infty(\Omega)$, we obtain from Arzelà-Ascoli theorem that there exists a subsequence $u_{d_k}$, $\{d_k\}\subset(0,+\infty)$ an increasing sequence, such that $u_{d_k}\to u_\infty$ uniformly in $\Omega$. Now, since from \eqref{intro2} 
    \begin{equation*}
        \Delta\parentheses{\frac{u_{d_k}}{P}}+\frac{1}{d_k}ru_{d_k}\parentheses{1-\frac{u_{d_k}}{K}}=0
    \end{equation*}
    for every $k\in\N$, take $k\to+\infty$ to obtain
    \begin{equation*}
        \Delta \parentheses{\frac{u_\infty}{P}}=0,
    \end{equation*}
    and since $\frac{\partial}{\partial n} (\frac{u_\infty}{P})=0$ in $\partial\Omega$, we obtain $u_\infty(x)=\beta P(x)$ for some constant $\beta \geq 0$.

    From \eqref{Eq.2 Feb25th} applied to $d=d_k$, let $k\to+\infty$ to obtain 
    \begin{equation*}
        0=\int_{\Omega}r(x)\beta P(x)\parentheses{1-\frac{\beta P(x)}{K(x)}}\,dx=\beta \int_\Omega r(x)P(x)\,dx-\beta ^2\int_\Omega\frac{r(x)}{K(x)}P(x)^2\,dx
    \end{equation*}
    which implies
    \begin{equation*}
        \beta=0\,\,\,\,\textrm{ or }\,\,\,\, \beta =\frac{\int_\Omega rP\,dx}{\int_\Omega\frac{r}{K}P^2\,dx}.
    \end{equation*}

    If we had $\beta =0$, then $u_{d_k}\to0$ on the $L^\infty(\overline{\Omega})$ norm, for there exists $k\in\N$ such that $1-\frac{u_{d_k}}{K}>0$ in $\overline{\Omega}$. Since $u_{d_k}>0$ on $\overline{\Omega}$, we then have
    \begin{equation*}
        \int_\Omega ru_{d_k}\parentheses{1-\frac{u_{d_k}}{K}}\,dx>0,
    \end{equation*}
    which contradicts \eqref{Eq.2 Feb25th}.
    
\end{proof}

\begin{corollary}
\label{cor:linear_dependence}
If $\displaystyle P = \alpha \frac{K}{r}$ for some $\alpha >0$ then
$$
\lim_{d \to +\infty} \int_{\Omega} u_d~dx = \int_{\Omega} K dx.
$$ 
\end{corollary}

\begin{proof}
    From Proposition \ref{Prop 2 DeAngelis}, $u_d\to \beta \frac{K}{r}$ in $L^\infty(\Omega)$ as $d\to+\infty$, where
\begin{equation*}
    \beta=\frac{\int_\Omega r\frac{K}{r}\,dx}{\int_\Omega\frac{r}{K}\brackets{\frac{K}{r}}^2\,dx}=\frac{\int_\Omega K\,dx}{\int_\Omega \frac{K}{r}\,dx}.
\end{equation*}
\end{proof}

%
%

\section{Main Results}
\label{sec:main}

\subsection{The choice of the dispersal strategy to guarantee abundance}

First of all, we explore the following question: can a species choose such a dispersal strategy $P$ that its total population exceeds the total carrying capacity?
The answer is positive, once $P$ is proportional to the ratio of the carrying capacity and the intrinsic growth rate, while this intrinsic growth rate is non-constant. In other words, the areas of attraction should have higher
available resources with smaller $r$.

\begin{theorem}
\label{Lou_generalize}
Let (A) be satisfied and $\displaystyle P = \alpha \frac{K}{r}$ for some $\alpha >0$, then the solution $u_d$ of \eqref{intro2} satisfies \eqref{intro1a} for every $d>0$.
\end{theorem}

\begin{proof}
    Setting $r=\alpha\frac{K}{P}$ on \eqref{intro2}, we obtain
    \begin{equation}\label{Eq1 r=K/P}
    \begin{cases}
        d\Delta\parentheses{\frac{u_d}{P}}+\frac{\alpha}{P}u_d\parentheses{K-u_d}=0 & \mbox {on }  \,\,\,\Omega, \\
        \frac{\partial}{\partial n}\parentheses{\frac{u_d}{P}}=0  & \mbox {on } \,\partial\Omega.
    \end{cases} 
\end{equation}
    Multiply Equation \eqref{Eq1 r=K/P} by $\frac{P}{u_d}$ and integrate over $\Omega$, to get
	\begin{equation}\label{Eq2 r=K/P}
		d\int_\Omega\Delta\parentheses{\frac{u_d}{P}}\frac{P}{u_d}\,dx+\alpha\int_\Omega(K-u_d)\,dx=0.
	\end{equation}
From the divergence theorem and the boundary condition of \eqref{Eq1 r=K/P}, we rewrite the first term of \eqref{Eq2 r=K/P} as
\begin{align*}
    \int_\Omega\Delta\parentheses{\frac{u_d}{P}}\frac{P}{u_d}\,dx&=-\int_\Omega \nabla\parentheses{\frac{u_d}{P}}\cdot\nabla\parentheses{\frac{P}{u_d}}\,dx=-\int_\Omega \nabla\parentheses{\frac{u_d}{P}}\cdot\nabla\parentheses{\parentheses{\frac{u_d}{P}}^{-1}}\,dx\\
    &=\int_\Omega\nabla\parentheses{\frac{u_d}{P}}\cdot\nabla\parentheses{\frac{u_d}{P}}\parentheses{\frac{u_d}{P}}^{-2}\,dx=\int_\Omega \module{\nabla\parentheses{\frac{u_d}{P}}}^2\parentheses{\frac{P}{u_d}}^2\,dx,
\end{align*}
and therefore, from \eqref{Eq2 r=K/P},
\begin{equation*}
    \alpha\int_\Omega(u_d-K)\,dx=d\int_\Omega \module{\nabla\parentheses{\frac{u_d}{P}}}^2\parentheses{\frac{P}{u_d}}^2\,dx>0.
\end{equation*}
The last integral could only be zero if the  ratio $\frac{u_d}{P}$ were constant over $\Omega$, leading to $u_d \equiv K$ being a positive solution to \eqref{Eq1 r=K/P} which is also proportional to $P$.
Then,  the growth rate $r=\alpha\frac{K}{P}$ would be constant over $\Omega$, which we assume to be false.
\end{proof}

\begin{remark}\label{remark:Lou generalized}
In the case of $P=\alpha\frac{K}{r}$, it follows from Lemma \ref{Prop_1_DeAngelis} and from Corollary \ref{cor:linear_dependence} that $\int_\Omega u_d\,dx$ converges to $\int_\Omega K\,dx$ when both $d\to0^+$ and $d\to+\infty$. And since $\int_\Omega u_d\,dx>\int_\Omega K\,dx$ for every $d>0$, the total population of solutions to \eqref{Eq1 r=K/P} should be maximized for an intermediate diffusion coefficient $d^*\in(0,+\infty)$. However, multiple local maxima for $\int_\Omega u_d\,dx$ might exist, as illustrated in Example \ref{example:tot pop not unimodal}.

The assumption in (A) that $r$ is non-constant is quite significant: if $r$ is constant, the strategy chosen in Theorem~\ref{Lou_generalize} will lead to $P$ being proportional to $K$, leading to a solution identical to $K$ for any dispersal coefficient. And, as Theorem~\ref{th_const_growth_rate} below illustrates, this case maximizes all the possible population averages for a constant growth rate.

\end{remark}

Note that the result in \cite{Lou} is a particular case of Theorem~\ref{Lou_generalize} corresponding to $r$ being proportional to $K$ and $P\equiv1$. Then, the winning strategy is free non-directed dispersal. 
If the intrinsic growth rate is not totally aligned with the carrying capacity, a choice of more sophisticated strategies is required.

The second question is what happens if the growth rate in the heterogeneous environment is constant.

\begin{theorem}
\label{th_const_growth_rate}
    Let $r>0$ be constant and $K$, $P>0$ be non-constant and linearly independent over $\Omega$. Then the solution $u_d$ of \eqref{intro2} satisfies
    \begin{equation*}
        \int_\Omega u_d\,dx<\int_\Omega K\,dx,~~\textrm{ for all }~~d>0.
    \end{equation*}
\end{theorem}

\begin{proof}
The proof literally repeats the argument in  \cite{Guo_2020} for $P \equiv 1$, but it is presented here for completeness.

    Assume the contrary that there exists some $d^*>0$ such that
    \begin{equation}\label{Eq.1 r constant}
        \int_\Omega u_{d^*}\,dx\geq\int_\Omega K\,dx.
    \end{equation}
    Integrating \eqref{intro2} for $d=d^*$, we get
    \begin{align}
        \nonumber&r\int_\Omega u_{d^*}\parentheses{1-\frac{u_{d^*}}{K}}\,dx=-d^*\int_\Omega\Delta\parentheses{\frac{u_{d^*}}{P}}\,dx=d^*\int_{\partial\Omega}\frac{\partial}{\partial n}\parentheses{\frac{u_{d^*}}{P}}\,dS=0\\
        \label{Eq.2 r constant}\implies&\int_\Omega u_{d^*}\parentheses{1-\frac{u_{d^*}}{K}}\,dx=0
    \end{align}
    from the boundary condition in \eqref{intro2}. Adding up \eqref{Eq.1 r constant} and \eqref{Eq.2 r constant}, we obtain
    \begin{equation*}
        0\leq\int_\Omega \parentheses{u_{d^*}- K+u_{d^*}\parentheses{1-\frac{u_{d^*}}{K}}}\,dx=\int_\Omega\parentheses{u_{d^*}-K}\parentheses{1-\frac{u_{d^*}}{K}}\,dx=-\int_\Omega\frac{1}{K}\parentheses{u_{d^*}-K}^2\,dx\leq0,
    \end{equation*}
    which leads to $\int_\Omega\frac{1}{K}\parentheses{u_{d^*}-K}^2\,dx=0$ and $u_{d^*}\equiv K$ over $\Omega$, which is only possible if $P\equiv K$ over $\Omega$, contradicting the assumption that $P$ and $K$ are linearly independent.

    Therefore, we must have $M(d)<\int_\Omega K\,dx$ for all $d>0$.
        
\end{proof}

\subsection{Tendencies for slow dispersal}

In this section, we consider solution asymptotics in the case $d \to 0^+$.

\begin{lemma}
 \label{Lemma 1}
 Suppose that (A) holds and let $u_d$ be a solution to \eqref{intro2}.
    \begin{enumerate}
        \item If $\displaystyle\int_\Omega\nabla\parentheses{\frac{K}{P}}\cdot\nabla \parentheses{\frac{1}{r}}\,dx<0$, then for $d>0$ small, 
        \begin{equation*}
            \int_\Omega u_d\,dx > \int_\Omega K\,dx.
        \end{equation*}

        \item If $\displaystyle\int_\Omega\nabla\parentheses{\frac{K}{P}}\cdot\nabla \parentheses{\frac{1}{r}}\,dx>0$ , then for $d>0$ small, 
        \begin{equation*}
            \int_\Omega u_d\,dx <\int_\Omega K\,dx.
        \end{equation*}

    \end{enumerate}
\end{lemma}

\begin{proof}
        Let us prove item~1, as item 2 is proven analogously. Multiply the first equation in  \eqref{intro2} by $\frac{K}{ru_d}$,
        \begin{equation*}
            d\Delta\parentheses{\frac{u_d}{P}}\frac{K}{ru_d}+K-u_d=0,
        \end{equation*}
        integrate it and use that $\frac{\partial}{\partial n}(\frac{u_d}{P})=0$ on $\partial\Omega$ to obtain
        \begin{equation}\label{Eq.1 Feb14th}
            -d\int_\Omega\nabla\parentheses{\frac{u_d}{P}}\cdot\nabla\parentheses{\frac{K}{ru_d}}\,dx+\int_\Omega(K-u_d)\,dx=0.
        \end{equation}
        
        Next, let us prove that
        \begin{equation*}
            \int_\Omega\module{\nabla\parentheses{\frac{u_d}{P}}\cdot\nabla\parentheses{\frac{K}{ru_d}}-\nabla\parentheses{\frac{K}{P}}\cdot\nabla\parentheses{\frac{1}{r}}}\,dx\to0 \textrm{ as } d\to0^+. 
        \end{equation*}
        For that purpose, notice that 
        \begin{align*}
            &\int_\Omega\module{\nabla\parentheses{\frac{u_d}{P}}\cdot\nabla\parentheses{\frac{K}{ru_d}}-\nabla\parentheses{\frac{K}{P}}\cdot\nabla\parentheses{\frac{1}{r}}}\,dx\\
            \leq&\int_\Omega\module{\nabla\parentheses{\frac{u_d}{P}-\frac{K}{P}}\cdot\nabla\parentheses{\frac{K}{ru_d}} }\,dx+\int_\Omega\module{\nabla\parentheses{\frac{K}{P}}\cdot\nabla\parentheses{\frac{K}{ru_d} -\frac{1}{r}}  }\,dx\\
            \leq&\norm{\nabla\parentheses{\frac{u_d}{P}-\frac{K}{P}}}_{L^2(\Omega)}\norm{\nabla\parentheses{\frac{K}{ru_d}}}_{L^2(\Omega)}+\norm{\nabla\parentheses{\frac{K}{P}}}_{L^\infty(\Omega)}\norm{\nabla\parentheses{\frac{K}{ru_d}-\frac{1}{r}}}_{L^1(\Omega)}\\
            =&I+II.
        \end{align*}
    From Proposition \ref{Prop_1_DeAngelis} we get that $u_d\to K$ on $L^\infty(\Omega)\cap W^{1,2}(\Omega)$ as $d\to0$, therefore the first term of $I$ vanishes. Moreover,
    \begin{equation*}
        \nabla\parentheses{\frac{K}{ru_d}}=\nabla\parentheses{\frac{K}{r}}\frac{1}{u_d}-\frac{K}{ru_d^2}\nabla u_d,
    \end{equation*}
    which is bounded, as $u_d\to K$ implies that $\nabla u_d$ is bounded and $u_d$ is far from zero almost everywhere. Therefore, $I\to0$ when $d\to0$.

    To estimate $II$, notice that
	\begin{align}
		\nonumber&\norm{\nabla\parentheses{\frac{K}{r u_d}-\frac{1}{r}}}_{L^1(\Omega)}\\
		\nonumber=&\int_\Omega\module{\frac{\nabla K}{ru_d}-\frac{K\nabla r}{r^2u_d}-\frac{K}{ru_d^2}\nabla u_d+\frac{\nabla r}{r^2}}\,dx\\
		\nonumber\leq&\int_\Omega\module{\frac{\nabla K}{ru_d}-\frac{\nabla u_d}{ru_d}}\,dx+\int_\Omega\module{\frac{\nabla u_d}{ru_d}-\frac{K}{ru_d^2}\nabla u_d}\,dx+\int_\Omega \module{\frac{K\nabla r}{r^2 u_d}-\frac{\nabla r}{r^2}}\,dx\\
		=&\int_\Omega\module{\frac{1}{ru_d}}\module{\nabla K-\nabla u_d}\,dx+\int_\Omega\module{\frac{\nabla u_d}{ru_d}}\module{1-\frac{K}{u_d}}\,dx+\int_\Omega\module{\frac{\nabla r}{r^2}} \module{\frac{K}{u_d}-1}\,dx.\label{Eq.2 Prop principal}
	\end{align}	
According to Proposition \ref{Prop_1_DeAngelis}, as $d\to0^+$, $u_d\to K$ on both $L^\infty(\Omega)$ and $W^{1,2}(\Omega)$ norms, therefore,  $|\frac{1}{u_d}|$ and $|\nabla u_d|$ are bounded on $\Omega$,      and therefore, each integral on \eqref{Eq.2 Prop principal} vanishes as $d\to0^+$.
    

    Therefore, as $d\to0^+$,
    \begin{equation*}
        \int_\Omega\nabla\parentheses{\frac{u_d}{P}}\cdot\nabla\parentheses{\frac{K}{ru_d}}\,dx=\int_\Omega\nabla\parentheses{\frac{K}{P}}\cdot\nabla\parentheses{\frac{1}{r}}\,dx+h(d),
    \end{equation*}
    where $\lim_{d\to0^+}h(d)=0$. And from \eqref{Eq.1 Feb14th}, we get
    \begin{equation*}
        \int_\Omega(K-u_d)\,dx=d\parentheses{\int_\Omega\nabla\parentheses{\frac{K}{P}}\cdot\nabla\parentheses{\frac{1}{r}}\,dx+h(d)}
    \end{equation*}
    which is negative for small enough $d>0$, since $\int_\Omega\nabla\parentheses{\frac{K}{P}}\cdot\nabla\parentheses{\frac{1}{r}}\,dx<0$.

\end{proof}


\begin{definition}
    We say that two functions $f:\Omega\to\R$ and $g:\Omega\to\R$ are {\em correlated by a function $h:\R\to\R$} in $\Omega$ if we have
    \begin{equation}\label{def correlations Eq1}
        f(x)=h(g(x)) \mbox{ ~~ for  all ~~ } x\in\Omega.
    \end{equation}
    If \eqref{def correlations Eq1} is satisfied in $\Omega$ and $h$ is non-decreasing, we say that $f$ and $g$ are {\em positively correlated} in $\Omega$, and if $h$ is non-increasing, we say that $f$ and $g$ are {\em negatively correlated} in $\Omega$.
\end{definition}

    \begin{corollary} 
\label{corollary:K_equals_r} 
    Suppose $r\equiv K$ over $\Omega$, and $P$ and $K$ are correlated by a function $h:(0,+\infty)\to(0,+\infty)$, and we write $P(x)=h(K(x))$ in $\Omega$. 
    \begin{enumerate}
        \item If $h'(t)>\frac{h(t)}{t}$ for all $t>0$, then for $d>0$ small,
        \begin{equation*}
            \int_\Omega u_d\,dx<\int_\Omega K\,dx.
        \end{equation*}

        \item If $h'(t)<\frac{h(t)}{t}$ for all $t>0$, then for $d>0$ small,
        \begin{equation*}
            \int_\Omega u_d\,dx>\int_\Omega K\,dx.
        \end{equation*}
    \end{enumerate}
\end{corollary}

\begin{proof}
    \begin{enumerate}
        \item By Lemma \ref{Lemma 1}, it is enough to determine the sign of 
        \begin{equation*}
            \int_\Omega \nabla\parentheses{\frac{K}{P}}\cdot\nabla \parentheses{\frac{1}{K}}\,dx=\int_\Omega\parentheses{-\frac{|\nabla K|^2}{PK^2}+K\frac{\nabla P\cdot\nabla K}{P^2K^2}}\,dx.
        \end{equation*}
        If $P(x)=h(K(x))$, then
        \begin{equation*}
            \nabla P=h'(K(x))\nabla K(x),
        \end{equation*}
        and we obtain
        \begin{equation*}
            \int_\Omega \nabla\parentheses{\frac{K}{P}}\cdot\nabla \parentheses{\frac{1}{K}}\,dx=\int_\Omega\parentheses{-P+Kh'(K)}\frac{|\nabla K|^2}{P^2K^2}\,dx=\int_\Omega\parentheses{-h(K)+Kh'(K)}\frac{|\nabla K|^2}{P^2K^2}\,dx.
        \end{equation*}
        If $h'(t)>\frac{h(t)}{t}$ for all $t>0$, then $-h(t)+th'(t)>0$ for all $t>0$, and $\int_\Omega \nabla\parentheses{\frac{K}{P}}\cdot\nabla \parentheses{\frac{1}{K}}\,dx>0$, which implies, by Lemma \ref{Lemma 1}, that
        \begin{equation*}
            \int_\Omega u_d\,dx<\int_\Omega K\,dx.
        \end{equation*}
        
        And if $h'(t)<\frac{h(t)}{t}$ for all $t>0$, then $-h(t)+th'(t)<0$, implying $\int_\Omega \nabla\parentheses{\frac{K}{P}}\cdot\nabla \parentheses{\frac{1}{K}}\,dx<0$, and follows from Lemma \ref{Lemma 1} that
        \begin{equation*}
            \int_\Omega u_d\,dx>\int_\Omega K\,dx.
        \end{equation*}
        
    \end{enumerate}
\end{proof}

A sufficient condition for Lemma \ref{Lemma 1} is when there exists a positive or negative correlation between $r$ and $\frac{K}{P}$, a fact we prove next.

\begin{theorem}\label{thm: correlation r and K/P}
    Assume assumption (A) holds. If $r$ and $\frac{K}{P}$ are positively correlated in $\Omega$, then, for $d>0$ small,
    \begin{equation*}
    	\int_\Omega u_d\,dx>\int_\Omega K\,dx.
    \end{equation*}
    
    If $r$ and $\frac{K}{P}$ are negatively correlated in $\Omega$, we get, for $d>0$ small,
    \begin{equation*}
    	\int_\Omega u_d\,dx<\int_\Omega K\,dx.
    \end{equation*}
     
\end{theorem}

\begin{proof}
If $r$ and $\frac{K}{P}$ are positively correlated (similar proof for negative correlation),  there exists $h:\R\to\R$ non-decreasing with $r(x)=h\parentheses{\frac{K(x)}{P(x)}}$ for every $x\in\Omega$, thus
    \begin{equation*}
        \nabla \parentheses{\frac{1}{r}}=-h\parentheses{\frac{K}{P}}^{-2}h'\parentheses{\frac{K}{P}}\nabla\parentheses{\frac{K}{P}}
    \end{equation*}
    and
    \begin{equation*}
        \nabla\parentheses{\frac{K}{P}}\cdot\nabla \parentheses{\frac{1}{r}}=-\frac{h'\parentheses{K/P}}{h\parentheses{K/P}^2}\module{\nabla\parentheses{\frac{K}{P}}}^2\leq0\,\textrm{ everywhere in }\Omega.
    \end{equation*}
    And since $h'(x)\equiv0$ or $\frac{K}{P}\equiv$ cte in $\Omega$ means that $r$ is constant, we obtain the strict inequality 
   \begin{equation*}
   \int_\Omega\nabla\parentheses{\frac{K}{P}}\cdot\nabla \parentheses{\frac{1}{r}}\,dx<0.
   \end{equation*}   
   The proof is concluded by applying Lemma \ref{Lemma 1}.
\end{proof}

\begin{remark}\label{rmk: tot pop maximized}
    When $\displaystyle\int_\Omega\nabla\parentheses{\frac{K}{P}}\cdot\nabla \parentheses{\frac{1}{r}}\,dx<0$, we can only be sure that $M(d)=\displaystyle\int_\Omega u_d\,dx$ is maximized at some intermediate dispersion coefficient $d^*\in(0,+\infty)$ when $M(0)\geq M(+\infty)$, or equivalently,
    \begin{equation*}
        \int_\Omega K\,dx\geq\frac{\int_\Omega rP\,dx}{\int_\Omega\frac{r}{K}P^2\,dx}\int_\Omega P\,dx.
    \end{equation*}
    If $M(0)<M(+\infty)$, it is possible for $M(d)$ to be always increasing over $(0,+\infty)$, as illustrated on Example \ref{example:tot pop always increasing}.  
\end{remark}

Some additional relations of weighted averages are presented in the Appendix.

\subsection{The case when the correlation is described by a power  function}

Based on Theorem \ref{Lou_generalize}, it is reasonable to question if solutions to \eqref{intro2} satisfy the estimate \eqref{intro1a} when, instead of $r=\alpha\frac{K}{P}$, we have
\begin{equation}\label{eq. r=(K/P)lambda}
    r=\alpha\parentheses{\frac{K}{P}}^\lambda~~~\textrm{in}~~~\Omega,
\end{equation}
where $\alpha\in(0,+\infty)$ and $\lambda\in\R$. If $\lambda>0$, $r$ and $\frac{K}{P}$ are correlated by the function $h(x)=\alpha x^\lambda$, which is increasing in $(0,+\infty)$, and we can apply Theorem \ref{thm: correlation r and K/P} to obtain an estimate \eqref{intro1a} for small values of $d>0$. 

In this section, let us denote by $u_{d,\lambda}$ the positive solution to
\begin{equation}\label{eq. prob r=(K/P)lambda}
    \begin{cases}
        d\Delta \parentheses{\frac{u(x)}{P(x)}} +\alpha\parentheses{\frac{K(x)}{P(x)}}^\lambda u (x) \parentheses{1-\frac{u(x)}{K(x)}}=0, &  x \in \Omega,\\
        \frac{\partial u}{\partial n}(x)=0 , &  x \in \partial\Omega,
    \end{cases} 
\end{equation}
for each $\lambda\in\R$ and $d>0$, and the total population function
\begin{equation*}
    M_\lambda(d)=\int_\Omega u_{d,\lambda}\,dx.
\end{equation*}

To analyze the total population of $u_{d,\lambda}$ as $d$ increases, let us follow the strategy of Remark \ref{rmk: tot pop maximized} and analyze the relation between $M_\lambda(0)$ and $M_\lambda(+\infty)$. 

\begin{lemma}
    Let $f$, $g:\Omega\to\R$ be positive functions. If $f$ is non-constant and $\int_\Omega g\,dx=1$, it holds
    \begin{equation}\label{Eq.2 M(infty) increasing}
        \frac{\int_\Omega gf\ln f\,dx}{\int_\Omega gf\,dx}>\int_\Omega g\ln f\,dx.
    \end{equation}
\end{lemma}

\begin{proof}
    Notice that
    \begin{align*}
        &\int_\Omega g(x)f(x)\ln f(x)\,dx-\int_\Omega g(x)\ln f(x)\,dx\int_\Omega g(y)f(y)\,dy\\
        =&\int_\Omega\parentheses{f(x)-\int_\Omega g(y)f(y)\,dy}\parentheses{\ln f(x)-\ln\parentheses{\int_\Omega g(y)f(y)\,dy}}g(x)\,dx,
    \end{align*}
    since
    \begin{align*}
        &\int_\Omega \parentheses{f(x)-\int_\Omega g(y)f(y)\,dy}\ln\parentheses{\int_\Omega g(y)f(y)\,dy}g(x)\,dx=0,
    \end{align*}
    where we recall that $\int_\Omega g(x)\,dx=1$. Therefore, denoting $L=\int_\Omega g(y)f(y)\,dy$, we obtain
    \begin{align*}
        \int_\Omega gf\ln f\,dx-\int_\Omega g\ln f\,dx\int_\Omega gf\,dy=\int_\Omega\parentheses{f(x)-L}\parentheses{\ln f(x)-\ln L}g(x)\,dx.
    \end{align*}
    We notice that this integral is non-negative, since $g(x)>0$, and $\ln y$ being increasing on $(0,+\infty)$ implies that $f(x)-L$ and $\ln f(x)-\ln L$ always have the same signal over $\Omega$, and therefore,
    \begin{equation*}
        \parentheses{f(x)-L}\parentheses{\ln f(x)-\ln L}\geq0~~~\textrm{for every}~~~x\in\Omega.
    \end{equation*}
    Finally, $\int_\Omega\parentheses{f(x)-L}\parentheses{\ln f(x)-\ln L}g(x)\,dx=0$ if, and only if, $f(x)=L$ for all $x\in\Omega$, which contradicts $f$ being non-constant. 
\end{proof}

\begin{theorem}\label{thm: M(infty) for r=(K/P)^beta}
    The mapping $\lambda\mapsto M_\lambda(+\infty)=\frac{\int_\Omega \frac{K^\lambda}{P^{\lambda-1}}\,dx}{\int_\Omega \frac{K^{\lambda-1}}{P^{\lambda-2}}\,dx}\int_\Omega P\,dx$ is a strictly increasing function of $\lambda$ in $\R$.
\end{theorem}
\begin{proof}
    It is enough to show that
    \begin{equation}\label{Eq.0 M(infty) increasing}
        F(\lambda)=\frac{\int_\Omega \frac{K^\lambda}{P^{\lambda-1}}\,dx}{\int_\Omega \frac{K^{\lambda-1}}{P^{\lambda-2}}\,dx}=\frac{\int_\Omega K\parentheses{\frac{K}{P}}^{\lambda-1}\,dx}{\int_\Omega K\parentheses{\frac{K}{P}}^{\lambda-2}\,dx}
    \end{equation}
    is strictly increasing in $\R$. Notice that
    \begin{equation*}
        \ln(F(\lambda))=\ln\parentheses{\int_\Omega K\parentheses{\frac{K}{P}}^{\lambda-1}\,dx}-\ln\parentheses{\int_\Omega K\parentheses{\frac{K}{P}}^{\lambda-2}\,dx},
    \end{equation*}
    and differentiating with respect to $\lambda$, 
    \begin{equation}\label{Eq.1 M(infty) increasing}
        \frac{F'(\lambda)}{F(\lambda)}=\frac{\int_\Omega K\parentheses{\frac{K}{P}}^{\lambda-1}\ln\parentheses{\frac{K}{P}}\,dx}{\int_\Omega K\parentheses{\frac{K}{P}}^{\lambda-1}\,dx}-\frac{\int_\Omega K\parentheses{\frac{K}{P}}^{\lambda-2}\ln\parentheses{\frac{K}{P}}\,dx}{\int_\Omega K\parentheses{\frac{K}{P}}^{\lambda-2}\,dx}.
    \end{equation}
    The right-hand side of Eq. \eqref{Eq.1 M(infty) increasing} is positive, from \eqref{Eq.2 M(infty) increasing} applied to 
    \begin{equation*}
       f=\frac{K}{P},~~~~g=\parentheses{\int_\Omega K\parentheses{\frac{K}{P}}^{\lambda-2}\,dx}^{-1}K\parentheses{\frac{K}{P}}^{\lambda-2} 
    \end{equation*}
     And since $F(\lambda)>0$, we obtain from \eqref{Eq.1 M(infty) increasing} that $F'(\lambda)>0$ in $\R$. Therefore, $F(\lambda)$ is increasing in $\R$.
\end{proof}

Theorem \ref{thm: M(infty) for r=(K/P)^beta} and the fact that  $M_\lambda(0)=M_\lambda(+\infty)$ for $\lambda=1$ implies that
\begin{equation*}
    M_\lambda(0)>M_\lambda(+\infty) \textrm{ for } \lambda>1
\end{equation*}
and 
\begin{equation*}
    M_\lambda(0)<M_\lambda(+\infty) \textrm{ for } \lambda<1.
\end{equation*}
These facts are elucidated in Example \ref{example:r(K/P)^lambda}, and we explore their implications for our understanding of the behaviour of $M_\lambda(d)$ in the Discussion section.

\section{Examples}
\label{sec:examples}

In this section, we illustrate some properties of the total population function $M(d)$ that were mentioned on Section \ref{sec:main}. Note that, in this Section, all simulations were performed on the one-dimensional domain $\Omega=(0,1)$.

\begin{example}\label{example:tot pop always increasing}
    This example explores the possible behaviours of $M(d)$ when $\int_\Omega \nabla\parentheses{\frac{K}{P}}\cdot \nabla \parentheses{\frac{1}{r}}\,dx<0$. We know from Lemma \ref{Lemma 1} that $M(d)$ is increasing near $d=0$, and from Remark \ref{rmk: tot pop maximized}, a global maximum is attained at $d^*\in(0,+\infty)$ if $M(0)>M(+\infty)$. This scenario is illustrated in Figure \ref{Fig tot pop increasing}a, which was generated using the parameters $K(x)=2+\cos(\pi x)$, $P(x)=1+\frac{1}{5}\cos(\pi x)$ and $r(x)=\frac{5}{4}+\frac{1}{4}\cos(\pi x)$.
    
    It is also possible for $M(d)$ to be strictly increasing in $(0,+\infty)$, without attaining a local maximum, when $M(0)<M(+\infty)$. For instance, considering $K(x)=2+\cos(\pi x)$, $P(x)=1+\frac{1}{5}\cos(\pi x)$ and $r(x)=\exp(4\cos(\pi x))$,  the total population increases monotonically with $d$, as shown in Figure \ref{Fig tot pop increasing}b.

    \begin{figure}[http]
   \centering
   \includegraphics[scale=0.5]{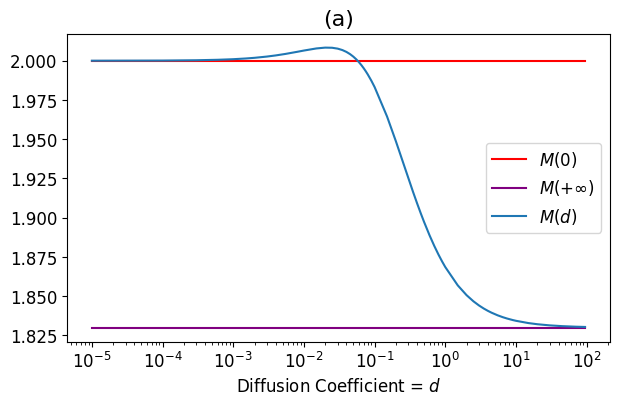}
   \includegraphics[scale=0.5]{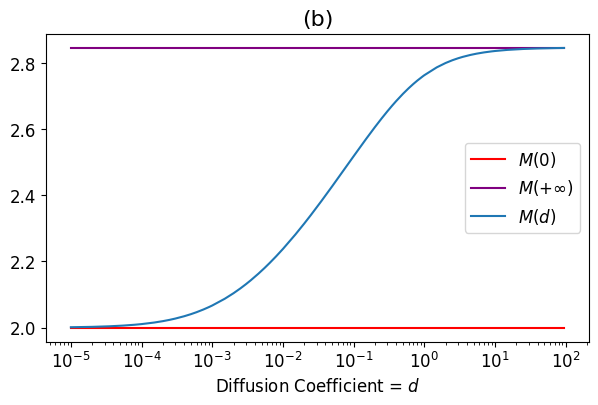}
   \caption{Different behaviours for the total population when $\int_\Omega \nabla\parentheses{\frac{K}{P}}\cdot \nabla r\,dx>0$, for $K(x)=2+\cos(\pi x)$, $P(x)=1+\frac{1}{5}\cos(\pi x)$ and (a) $r(x)=\frac{5}{4}+\frac{1}{4}\cos(\pi x)$, (b)  $r(x)=\exp(4\cos(\pi x))$.}
   \label{Fig tot pop increasing}
\end{figure}

\end{example}
 
\begin{example}\label{example:tot pop not unimodal}
Returning to the case where $P$ and $\frac{K}{r}$ are proportional, recall from Remark \ref{remark:Lou generalized} that $M(d)$ exceeds $\int_\Omega K\,dx$ for all $d>0$, with $M(d)$ approaching this integral as $d\to0^+$ and $d\to+\infty$, implying that $M(d)$ attains a maximum value at $d^*\in(0,+\infty)$. 

We now illustrate the local maximum value of $M(d)$ my not be unique. Let $K(x)=0.1+\cos(\pi x)+5\cos^2(\pi x)-2\cos^3(\pi x)$, $P(x)=1.5-3\cos(\pi x)+\cos^2(\pi x)+3\cos^6(\pi x)$ and $r(x)=\frac{K(x)}{P(x)}$. As shown in Figure \ref{Figure: not unimodal tot pop}, the corresponding total population $M(d)$ exhibits two different local maxima.

    \begin{figure}[h]
   \centering
   \includegraphics[scale=0.5]{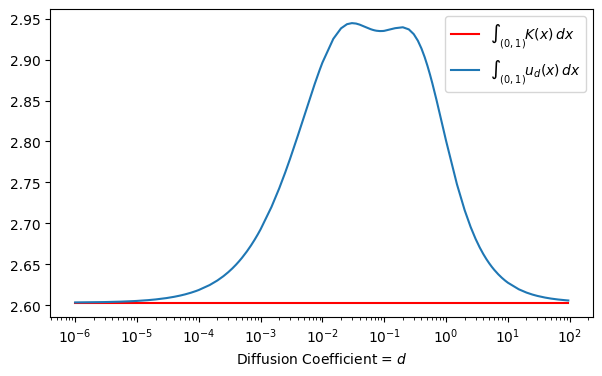}
   \caption{Non-unimodal total population curve under spatially heterogeneous dispersal strategy, for $K(x)=0.1+\cos(\pi x)+5\cos^2(\pi x)-2\cos^3(\pi x)$, $P(x)=1.5-3\cos(\pi x)+\cos^2(\pi x)+3\cos^6(\pi x)$ and $r(x)=\frac{K(x)}{P(x)}$.}
   \label{Figure: not unimodal tot pop}
\end{figure}

Investigations are being conducted to examine the necessary conditions for the parameters $P$ and $K$ for which $M(d)$ changes monotonicity exactly once and attains one unique maximum value.
\end{example}

The last example illustrates the result of Theorem \ref{thm: M(infty) for r=(K/P)^beta}, which analyzes how the total population depends on the parameter $\lambda$ present in the growth rate
$r(x)=\parentheses{\frac{K(x)}{P(x)}}^\lambda$.

\begin{example}
\label{example:r(K/P)^lambda}
Let $K(x)=2+\cos(\pi x)$, $P(x)=2-\cos(2\pi x)$, and $r(x)=\parentheses{\frac{K(x)}{P(x)}}^\lambda$. The parameter $\lambda$ is varied, taking the values $\lambda=-1$, $0$, $0.5$, $1$, $1.4$ and $2.3$ for comparison.
Figure \ref{Fig r=(K/P)^lambda} shows the corresponding total population curves $M_\lambda(d)$.

\begin{figure}[http!]     
   \centering
   \includegraphics[width=0.12\textwidth]{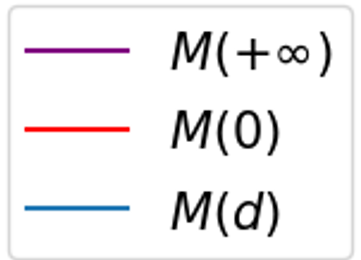}
   
   \includegraphics[width=0.45\textwidth]{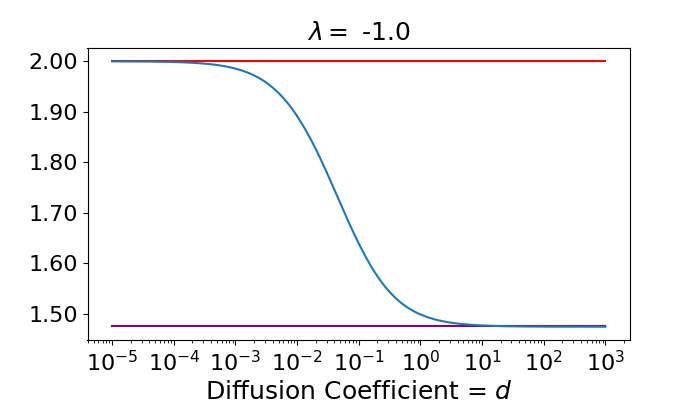}
   \includegraphics[width=0.45\textwidth]{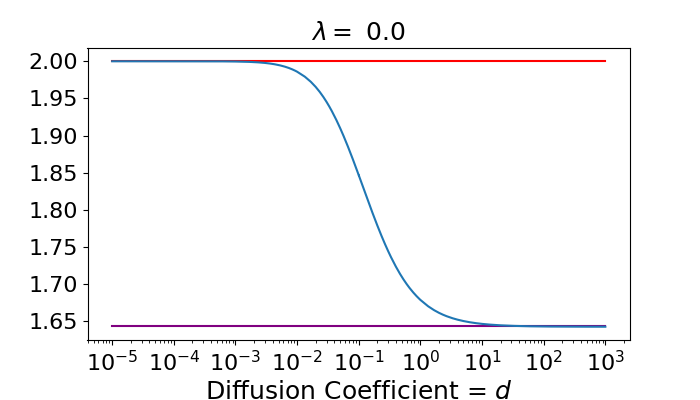}
   
   \includegraphics[width=0.45\textwidth]{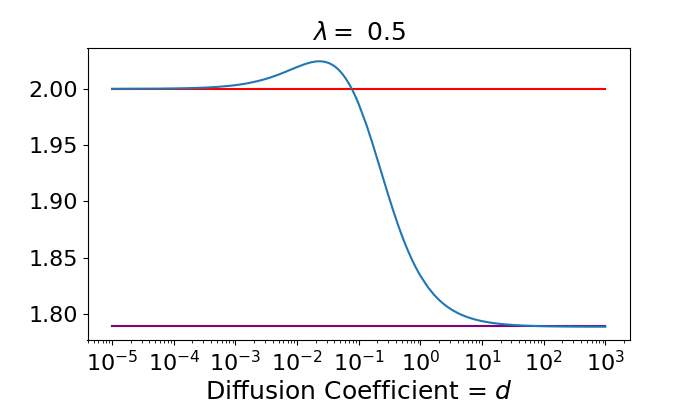}
    \includegraphics[width=0.45\textwidth]{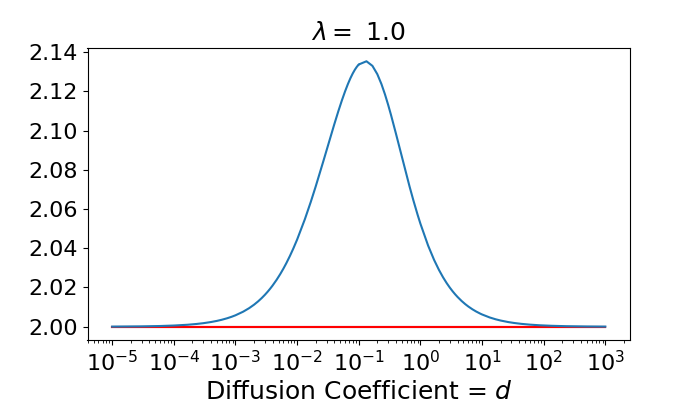}
    
    \includegraphics[width=0.45\textwidth]{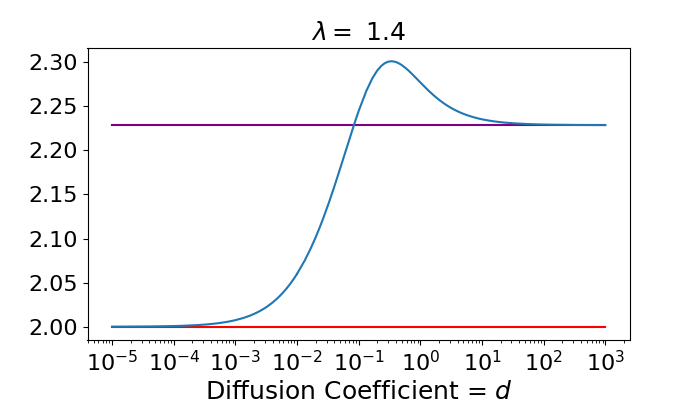}
    \includegraphics[width=0.45\textwidth]{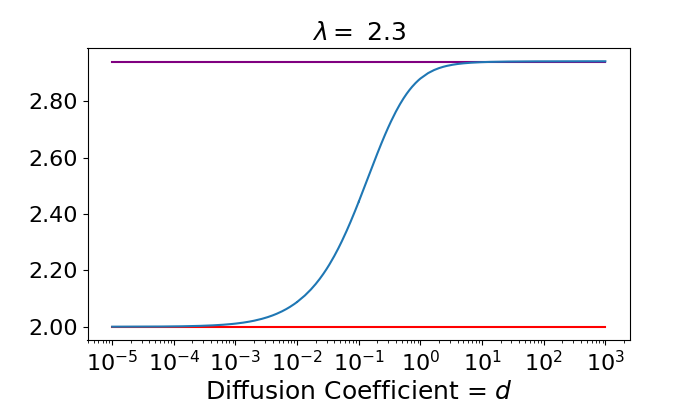}

   \caption{Total population curves for $K(x)=2+\cos(\pi x)$, $P(x)=2-\cos(2\pi x)$, and $r(x)=\parentheses{\frac{K(x)}{P(x)}}^\lambda$, where $\lambda=-1$, $0$, $0.5$, $1$, $1.4$ and $2.3$}\
   \label{Fig r=(K/P)^lambda}
\end{figure}

First, with very slow diffusion, the same limit total population of $M_\lambda(0)=2$ is observed in every figure, while the fast dispersal limit $M_\lambda(+\infty)$ varies as an increasing function of $\lambda$, in agreement with Theorem \ref{thm: M(infty) for r=(K/P)^beta}. For  $\lambda=-1$, we obtain negative correlation between $r$ and $\frac{K}{P}$, and $M(d)$ is a decreasing sigmoid-type function, which is in accordance with Theorem \ref{thm: correlation r and K/P}, and converges to $M_\lambda(+\infty)<M_\lambda(0)$ as $d\to+\infty$. The same behaviour of $M_\lambda(d)$ is observed for $\lambda=0$, leading to $M_\lambda(d)<\int_\Omega K\,dx$ for all $d>0$, as stated by \cite{Guo_2020}. 

For $\lambda=0.5$, the function $M_\lambda(d)$ takes a unimodal form, initially increasing due to the positive correlation between $r$ and $\frac{K}{P}$ (Theorem \ref{thm: correlation r and K/P}), reaching a maximum value, and eventually decreasing to a value below $M_\lambda(0)$.
For $\lambda=1$, $M_\lambda(d)$ is always greater than $\int_\Omega K\,dx$ with equal limit values as $d \to 0^+$ and $d\to+\infty$, thereby validating Theorem \ref{Lou_generalize} and Remark \ref{remark:Lou generalized}.

For $\lambda >1$, $r$ and $\frac{K}{P}$ are positively correlated, and Theorem \ref{thm: M(infty) for r=(K/P)^beta} guarantees that the limit value $M_\lambda(+\infty)$ exceeds $M_\lambda(0^+)$.  The transient behaviour, however, differs for $\lambda=1.4$ and $\lambda=2.3$, as for the first case $M_\lambda(d)$ is unimodal, attaining a maximum value in $(0,+\infty)$, whereas for $\lambda=2.3$, $M_\lambda(d)$ takes the monotone increasing sigmoid form.

Lastly, comparing all simulations, we observe the maximum and minimum population levels are non-decreasing functions of $\lambda$. Further analysis is underway to determine if this relation holds universally.
\end{example}

\section{Discussion}
\label{sec:discussion}

We start with briefly  summarizing the findings of the paper. 

\begin{enumerate}
\item
We identified the case when diffusion leads to population increase, for any of its values.
As a special case for $P \equiv 1$, this includes Lou's result in \cite{Lou}.

\item
Assuming positive (or negative) correlation between the growth rate  and the ratio of the carrying capacity to the dispersal strategy, 
which for high levels of diffusion characterizes per capita space-dependent available resources,
we determine when a small diffusion has a positive (negative) effect on the total population size compared to its immobile counterpart.

\item
In the particular case when the growth rate is a power function of the ratio $r=(K/P)^{\lambda}$, we can characterize the global population size for large dispersal speed.
We conclude that, once the species adopt a large enough diffusion rate, the total population exceeds the total of the carrying capacity, once 
the power function is convex, and is less if the function is concave. This fact has never been established before for the logistic equation with a regular diffusion. 
The two cases that have been already investigated in the particular case $P \equiv 1$ are $\lambda=1$  \cite{Lou} and $\lambda=0$ \cite{Guo_2020}.
We illustrate that the values   $\lambda=0$ and $\lambda=1$ are in some sense critical.
\end{enumerate}

Let us comment on the third item. Note that the intermediate behaviour of the total population size as a function of the diffusion coefficient exhibits distinct profiles depending on the value of $\lambda$, as illustrated in Fig.~\ref{Fig r=(K/P)^lambda}. Based on the results of this work, we conclude that for the case where $r=(K/P)^{\lambda}$, the following holds:


\begin{enumerate}
\item
For $\lambda\leq0$, the average population  decreases from its carrying capacity at $d=0$.
For $\lambda<0$, this is justified from $r=\alpha\parentheses{\frac{K}{P}}^\lambda$ representing a negative correlation between $r$ and $\frac{K}{P}$ and from Theorem \ref{thm: correlation r and K/P}, while the $\lambda=0$ case is covered by \cite{Guo_2020}. Furthermore,  since Theorem \ref{thm: M(infty) for r=(K/P)^beta} gives $M_\lambda(+\infty)<M_\lambda(0)$, we expect that $M_\lambda(d) < M_\lambda(0)$ for all $d\in(0,+\infty)$.

\item For $0<\lambda<1$,
$r=\alpha\parentheses{\frac{K}{P}}^\lambda$ leads to a positive correlation between $r$ and $\frac{K}{P}$, and Theorem \ref{thm: correlation r and K/P} implies that the average population exceeds the average carrying capacity for slow diffusion $d>0$. However, the inequality $M_\lambda(d)>\int_\Omega K\,dx$ cannot be satisfied for all $d>0$, specifically for $d$ large enough, since the total population of $u_{d,\lambda}$ converges to $M_\lambda(+\infty)<\int_\Omega K\,dx$ as $d\to+\infty$. Therefore, $M_\lambda(d)$ is maximized for some diffusion coefficient $d^*\in(0,+\infty)$.
\item For $\lambda=1$, Theorem~\ref{Lou_generalize} asserts that the total population exceeds the total carrying capacity for every $d>0$.
Due to the continuity of solutions to \eqref{eq. prob r=(K/P)lambda} with respect to $\lambda$, we expect that for values of $\lambda$ close to $1$, the average population $M_\lambda(d)$ is still maximized at some intermediate point $d_\lambda^*\in(0,+\infty)$.

\item
For $\lambda>1$ large enough, the asymptotic value $M_\lambda(+\infty)$ should also become large enough that $M_\lambda(d)$ is an increasing function of $d$ in $(0,+\infty)$.
\end{enumerate}

Based on the preceding analysis, we formulate the following conjecture:
\begin{conjecture*}
    Assume $K$, $P$ and $r=\alpha\parentheses{\frac{K}{P}}^\lambda$ satisfy hypothesis (A). Then there exists a critical value $\lambda^*>1$ such that 
    \begin{itemize}
        \item[a)] If $\lambda\in(0,\lambda^*)$, the function $M_\lambda(d)$ attains its maximum at a finite diffusion rate $d^*_\lambda\in(0,+\infty)$;
        
        \item[b)] If $\lambda\in(\lambda^*,+\infty)$, the function $M_\lambda(d)$ is strictly increasing in $(0,+\infty)$, and thus does not attain a maximum value.  
    \end{itemize}
\end{conjecture*}

With these empirical results and some statements justified, especially in the asymptotic cases of very slow or very fast dispersal, the problem
of the qualitative description of the population size as a function of the diffusion coefficient still leaves many questions without answer.
Let us discuss some topics emerging from the research of the current paper.

\begin{enumerate}
\item
The first group of questions, as mentioned above, is related to the form of the total population as a function of the diffusion coefficient. 
\begin{enumerate}
\item
Investigation of the case $r=f (K/P)$, where $f(x)=x^{\lambda}$, opens ways to the study of the dependency of the population on the diffusion coefficient for
$f(x)=ax+b$, fractional linear, exponential (in this case, we expect the behaviour similar to  $f(x)=x^{\lambda}$ with large positive $\lambda$) etc. 
\item
What are the relations of $P,r,K$ leading to the monotone or the unimodal dependency of the total population on the diffusion coefficient?
\item
Prove or disprove the conjecture that, once $r=f(K/P)$, where $f$ is an increasing function, the total population for large diffusion exceeds the total carrying capacity for 
the convex $f$ and is less for concave $f$. Note that the fact of equality for linear $h$  follows from Corollary~\ref{cor:linear_dependence}.
\end{enumerate}
\item
Maximizing the total population size is not directly connected to the success in the compertition, or evolution stability, compare, for example, \cite{CCBook,Korobenko2012,Korobenko2014,Lou}.
However, there is a process which is closely connected to the population size - harvesting. Consider for the model
\begin{equation}
\label{harvesting}
    \begin{cases}
        d\Delta\parentheses{\frac{u}{P}}+ru\parentheses{1-\frac{u}{K}} - Eu=0, & x \in \Omega,  \\
        \frac{\partial}{\partial n}\parentheses{\frac{u}{P}}=0,  & x \in \partial\Omega
    \end{cases} 
\end{equation} 
the problem of maximizing the Maximum Sustainable Yield (MSY)  $\displaystyle   \int_{\Omega} E(x)u(x)~dx$. 
In \cite{Brav} for $P \equiv K$, it was noticed that $E=r/2$ leads to MSY. With some non-constant $P \not\equiv K$,
for which $r$ is MSY achieved? If, in addition to $E$, we can control the diffusion coefficient $d$, how is MSY 
$d$-dependent? 

\item
We explored the dependency of the average solution on $\lambda$ when
$r=(K/P)^{\lambda}$. Let us contemplate whether some of the observations in numerical examples can be justified theoretically.
\begin{enumerate}
\item
Is the maximum population value increasing as the function of $\lambda$, not only the limit population value for very fast dispersals?
\item
There is analysis of the limit $M_{\lambda}(+\infty)$. Has this function a limit for $\lambda \to +\infty$?
The same question can be considered for the supremum value of the population. 
However, if the conjecture is justified, for large $\lambda$, the value of $M_{\lambda}(+\infty)$ is also the supremum value.
\item
 In our analysis, we fixed $P$ and $K$ and got variable $r$. Is there any differece if $r,K$ are fixed, and we modify 
$\displaystyle  P=K/r^{1/\lambda}$?
\end{enumerate}
\end{enumerate}

\section*{Acknowledgment}

The authors  were  partially supported by the NSERC Discovery Grant \#  RGPIN-2020-03934.

\section{Appendix: Some Additional Inequalities on Weighted Averages}

In the Appendix, we consider some weighted averages that relate the total population with the parameters of Equation~\eqref{intro2}.

\begin{theorem}
     If $r$ and $K$ are positively correlated and  $\displaystyle\int_\Omega\nabla K\cdot \nabla P\,dx<0$, we get for small values of $d>0$ the weighted inequality
        \begin{equation*}
            \int_\Omega Pu_d\,dx>\int PK\,dx.
        \end{equation*}
\end{theorem}

\begin{proof}
    Multiply \eqref{intro2} by $\frac{KP}{r u_d}$ and integrate, to obtain
        \begin{equation*}
            -d\int_\Omega\nabla\parentheses{\frac{u_d}{P}}\cdot\nabla\parentheses{\frac{KP}{r u_d}}\,dx+\int_\Omega(PK-Pu_d)\,dx=0.
        \end{equation*}
        Then prove that
        \begin{equation*}
            \int\module{\nabla\parentheses{\frac{u_d}{P}}\cdot\nabla\parentheses{\frac{KP}{r u_d}}-\nabla\parentheses{\frac{K}{P}}\cdot\nabla\parentheses{\frac{P}{r}}}\,dx\to0 \textrm{ as } d\to0, 
        \end{equation*}
        as done on the proof of Lemma \eqref{Lemma 1}, and use that
        \begin{align*}
            \int_\Omega\nabla\parentheses{\frac{K}{P}}\cdot\nabla\parentheses{\frac{P}{r}}\,dx& =\int_\Omega\frac{P\nabla K-K\nabla P}{P^2}\cdot\frac{r\nabla P-P\nabla r}{r^2}\,dx
            \\ & =\int_\Omega\frac{P[r\nabla K+K\nabla r]\cdot\nabla P-Kr|\nabla P|^2-P^2\nabla K\cdot\nabla r}{P^2r^2}\,dx<0,
        \end{align*}
\end{proof}

\begin{theorem}\label{thm: weighted ineq 2}
    Suppose that $K$ and $P$ are linearly independent and are correlated by a function $h$, i.e. $P=h(K)$. 
    
    If $h'(t)\geq\frac{h(t)}{t}$ for $t\in(0,+\infty)$, then for $d>0$ small,
    \begin{equation*}
        \int_\Omega ru_d^2\,dx>\int_\Omega rKu_d\,dx.
    \end{equation*}

    And if $h'(t)\leq\frac{h(t)}{t}$ for $t\in(0,+\infty)$, then for $d>0$ small,
    \begin{equation*}
        \int_\Omega ru_d^2\,dx<\int_\Omega rKu_d\,dx 
<\int_\Omega rK^2\,dx.
    \end{equation*}
   
\end{theorem}

\begin{proof}
        Multiply \eqref{intro2} by $K$ and integrate:
        \begin{equation*}
            -d\int_\Omega\nabla\parentheses{\frac{u_d}{P}}\cdot\nabla K\,dx+\int_\Omega ru_d(K-u_d)\,dx=0.
        \end{equation*}
        Notice that
        \begin{equation*}
            \int\module{\nabla\parentheses{\frac{u_d}{P}}\cdot\nabla K-\nabla\parentheses{\frac{K}{P}}\cdot\nabla K}\,dx\to0 \textrm{ as } d\to0, 
        \end{equation*}
        which implies
        \begin{equation*}
            \int_\Omega ru_d(K-u_d)\,dx=d\parentheses{\int_\Omega\nabla\parentheses{\frac{K}{P}}\cdot\nabla K\,dx+o(1)}
        \end{equation*}
        and it is enough to use
        \begin{equation*}
            \nabla\parentheses{\frac{K}{P}}=\frac{P\nabla K-K\nabla P}{P^2}=\frac{h(K)\nabla K-Kh'(K)\nabla K}{P^2}
\end{equation*}
      \begin{equation*}
\implies \nabla\parentheses{\frac{K}{P}}\cdot\nabla K=(h(K)-Kh'(K))\frac{|\nabla K|^2}{P^2}\leq(\geq)0.
        \end{equation*}
In the case $h'(t)\leq\frac{h(t)}{t}$, we find, for small values of $d>0$,
\begin{equation}\label{Eq.1 thm weig ineq 2}
    \int_\Omega rKu_d\,dx>\int_\Omega ru_d^2\,dx
\end{equation}
From the Cauchy–Schwarz inequality, 
\begin{equation}\label{Eq.2 thm weig ineq 2}
    \parentheses{\int_\Omega rKu_d\,dx}^2=\parentheses{\int_\Omega r^\frac{1}{2}Kr^\frac{1}{2}u_d\,dx}^2\leq\int_\Omega rK^2\,dx\int_\Omega ru_d^2\,dx.
\end{equation}
Finally, \eqref{Eq.1 thm weig ineq 2} and \eqref{Eq.2 thm weig ineq 2} lead to
\begin{align*}
\int_\Omega rK u_d\,dx ~\int_\Omega r K u_d\,dx  & \leq\int_\Omega rK^2\,dx ~\int_\Omega ru_d^2\,dx 
\\
&  <  \int_\Omega rK^2\,dx   \int_\Omega rKu_d\,dx  
\end{align*}
hence, for slow diffusion,
\begin{equation*}
\int_\Omega r K u_d\,dx   
<  \int_\Omega rK^2\,dx,
\end{equation*}
which concludes the proof.

\end{proof}

\begin{remark}
    Note that if both inequalities $h'(t)\geq\frac{h(t)}{t}$ and $h'(t)\leq\frac{h(t)}{t}$ are satisfied for all $t\in(0,+\infty)$, then we must have the equality
    \begin{equation*}
        \frac{h'(t)}{h(t)}=\frac{1}{t}\implies\ln|h(t)|=\ln|t|+\alpha \implies h(t)=\alpha t, 
    \end{equation*}
    for some constant $\alpha >0$ and for all $t>0$, leading to $P=h(K)=\alpha K$. Since $K$ and $P$ are not linearly dependent, we conclude that the two items from Theorem \ref{thm: weighted ineq 2} cannot hold simultaneously, and there is no contradiction in the theorem's statement.
\end{remark}

\end{document}